\documentclass{proc-l}

\usepackage{amssymb}
\usepackage[breaklinks]{hyperref}
\usepackage{graphicx}

\newtheorem{theorem}{Theorem}[section]

\newtheorem{proposition}[theorem]{Proposition}
\theoremstyle{definition}

\newtheorem{example}[theorem]{Example}

\newtheorem{corollary}[theorem]{Corollary}
\usepackage{color}

\theoremstyle{remark}
\newtheorem{remark}[theorem]{Remark}

\def\sym{\text{Sym}}
\def\trace{\text{Trace}}
\def\R{\mathbb R}

 \DeclareMathOperator*{\argmin}{arg\,min}

\numberwithin{equation}{section}

\begin{document}

\title[ Omnibus CLTs for Fr\'echet means]{Omnibus CLTs for Fr\'echet means and nonparametric inference on non-Euclidean spaces}


\author{Rabi Bhattacharya}
\address{Department of Mathematics, The University of Arizona, Tucson, AZ, 85721}
\curraddr{}
\email{rabi@math.arizona.edu}

\author{Lizhen Lin}
\address{Department of Statistics and Data Sciences, The University of Texas at Austin, Austin, TX, 78712}
\curraddr{}
\email{lizhen.lin@austin.utexas.edu}
\thanks{}


\date{}

\dedicatory{}

\commby{Mark M. Meerschaert}

\begin{abstract}
  Two central  limit theorems for sample Fr\'echet means are derived, both significant for nonparametric inference on non-Euclidean spaces.  The first one, Theorem \ref{th-clt}, encompasses and improves upon most earlier CLTs on Fr\'echet means and broadens the scope of the methodology beyond manifolds to diverse new non-Euclidean data including those on certain stratified spaces which are important in the study of phylogenetic trees.  It does not require that the underlying distribution $Q$ have a density, and applies to both intrinsic and extrinsic analysis.  The second theorem, Theorem \ref{th-2.4}, focuses on intrinsic means on Riemannian manifolds of dimensions $d>2$ and breaks new ground by providing a broad CLT without any of the earlier restrictive support assumptions.  It makes the statistically reasonable assumption of a somewhat smooth density of $Q$.  The excluded case of dimension $d=2 $ proves to be an enigma, although the first theorem does provide a CLT in this case as well under a support restriction. Theorem \ref{th-2.4} immediately applies to spheres $S^d$, $d>2$, which are also of  considerable importance in applications to axial spaces and to landmarks based image analysis, as these spaces are quotients of spheres under a Lie group $\mathcal G $ of isometries of $S^d$.  
  
\end{abstract}


\maketitle

\section{Introduction}

The present article focuses on the nonparametric, or model independent, statistical analysis of manifold-valued and other non-Euclidean data that arise in many areas of science and technology.  The basic idea is to use means for comparisons among distributions, as one does with Euclidean data.  On a metric space $(S, \rho)$ there is a notion of the mean $\mu$ of a distribution $Q$,  perhaps first formulated in detail in \cite{frech}, as the minimizer of the expected squared distance from a point,
 \begin{equation}
\label{eq-frechetf}
    \mu = \argmin_{p} \int \rho^2(p, q) Q(dq),
    \end{equation}
assuming the integral is finite (for some $p$) and the minimizer is unique, in which case one  says that the \emph{Fr\'echet mean of $Q$ exists}.  This $\mu$ is called the \emph{Fr\'echet mean of $Q$}.  In general, the set of minimizers is called the \emph{Fr\'echet mean set} of $Q$, denoted $C_Q$.   It turns out that uniqueness is crucial for making comparisons among distributions. Usually the minimizer is unique under relatively minor restrictions, if the distance $\rho$ is the Euclidean distance inherited by the embedding $J$ of a $d$-dimensional manifold $M$ in an Euclidean space  $E^N$, such that $J(M)$ is closed.  Indeed, under the relabeling of $M$ by $J(M)$,  the Fr\'echet mean set in this case is given by
\begin{equation}
\argmin_{p \in J(M)} \|p-m(Q\circ J^{-1})\|^2,
\label{eq-frechet2}   
\end{equation}
where $\|x\|$ is the Euclidean norm on $E^N$ and $m(Q \circ J^{-1})$ is the usual Euclidean mean of the induced distribution $Q\circ J^{-1}$ on $E^N$.  Thus the minimizer is unique if and only if the projection of the Euclidean mean on the image $J(M)$ of $M$ is unique, in which case it is called an \emph{extrinsic mean}.    On the other hand, if $\rho_g$ is the geodesic distance on a Riemannian manifold $M$ with metric tensor $g$ having positive sectional curvature (in some region of $ M$), then conditions for uniqueness are known only for $Q$ with support in a relatively small geodesic ball   \cite{afsari11, karcher, kenws}, which is too restrictive an assumption from the point of view of statistical applications.    If the Fr\'echet mean exists under $\rho_g$ it is called  the \emph{intrinsic mean}. A complete characterization of uniqueness of \eqref{eq-frechetf} for $\rho=\rho_g$ on the circle $S^1$ for probabilities $Q$ with a continuous density (\cite{rabinotes}, \cite{rabibook}) indicates that the intrinsic mean exists broadly, without any support restrictions, if $Q$ has a smooth density. 

      An important question that arises in the use of  Fr\'echet means in nonparametric statistics is the choice of the distance $\rho$ on $M$. There are in general uncountably many embeddings $J$ and metric tensors $g$ on a manifold $M$. For intrinsic analysis there are often natural choices for the metric tensor $g$.  A good choice for extrinsic analysis is to  find an embedding $J$: $M\rightarrow E^N$ with $J(M)$ closed, which is \emph{equivariant} under a large Lie group $\mathcal G$ of actions on $M$. This means that there is a homomorphism  $g\rightarrow \Phi_g$ on $\mathcal G$ into the general linear group $GL(N, \R)$ such that $J\circ g = \Phi_g\circ J$ $\forall g \in \mathcal G.$ Such embeddings and extrinsic means under them have been derived for Kendall type shape spaces in \cite{rabivic03}, \cite{rabivic05}, \cite{vic05}, \cite{ananda},  \cite{dryen2}, and \cite{abs2}.  In most data examples that have been analyzed, using a natural metric tensor $g $ and an equivariant $J$ under a large group $\mathcal G$, the sample intrinsic and extrinsic means are virtually indistinguishable and the inference based on the two different methodologies yield almost identical results \cite{rabibook}. This provides an affirmation of good choices of distances.  It also strongly suggests that the intrinsic mean is unique in many-perhaps most-statistical applications.

     Our focus in this article is to provide the asymptotic distribution theory which is the basis of nonparametric inference based on Fr\'echet means. The omnibus CLT Theorem \ref{th-clt} implies earlier results on CLT's  and, in particular  extends them to certain stratified spaces. Unfortunately, for the intrinsic CLT a support condition is still needed for the theorem to apply. In Section 3 we remove these support conditions for CLT's on $S^d$, $d>2$, assuming statistically reasonable smooth densities. The implications of these results for axial spaces and Kendall's shape spaces,etc, are indicated. 
     
     Finally, it is important to distinguish the intrinsic mean on a Riemannian manifold $(M,g)$ from the \emph{Karcher mean} of $Q$ which minimizes the Fr\'echet function restricted to an open set $S$ containing the support of $Q$.


\section{An omnibus CLT for the Fr\'echet mean}
\label{sec-2}

Let $(S, \rho)$ be a metric space and $Q$  a probability measure on its Borel $\sigma$-field.  Define the \emph{Fr\'echet function} of $Q$ as
\begin{equation}
\label{eq-frechet}
	F(p) = \int \rho^2(p,q) Q(dq) \;  (p \in S).
\end{equation}
Assume that $F$ is finite on $S$ and has a unique minimizer $\mu = \argmin_p F(p)$. Then $\mu$ is called the \emph{Fr\'echet mean} of $Q$ (with respect to the distance $\rho$).  Under broad conditions, the Fr\'echet  sample mean $\mu_n$ of the empirical distribution $Q_n= \dfrac{1}{n}\sum_{j=1}^n \delta_{Y_j}$ based on independent $S$-valued random variables $Y_j$ ($j=1,\ldots, n$) with common distribution $Q$ is a consistent estimator of $\mu$. That is, $\mu_n\rightarrow \mu$ almost surely, as $n\rightarrow \infty$.  Here $\mu_n$ may be taken to be any measurable selection from the (random) set of minimizers of the Fr\'echet  function of $Q_n$, namely, $F_n(p) = \dfrac{1}{n}\sum_{j=1}^n  \rho^2(p,Y_j)$ (See \cite{ziezold}, \cite{rabivic03}, \cite{rabivic05} and \cite{rabibook}).

We make the following assumptions.

\begin{itemize}
\item [(A1)] (Uniqueness of $\mu$) The Fr\'echet mean $\mu$ of $Q$ is unique.

\item [(A2)] $\mu\in G$,  where $G$ is a measurable subset of $S$, and there is a homeomorphism $\phi : G\rightarrow U$, where $U$ is  an open subset of $\mathbb R^s$ for some $s\geq 1$ and $G$  is given its relative topology on $S$.  Also,
\begin{equation}
\label{eq-asa2}
x\mapsto h(x;q) : = \rho^2(\phi^{-1}(x),q)\;
\end{equation}
is twice continuously differentiable on $U$, for every $q$  outside a $Q$-null set.
\item[(A3)]  $P(\mu_n\in G)\rightarrow 1$ as $n\rightarrow \infty$.

\item[(A4)]  Let $D_rh(x;q) = \partial h(x;q)/\partial x_r, $  $D_{r, r'}=D_rD_{r'}$,  $1\leq r, r' \leq s$. Then
\begin{equation}
	E | D_rh (\phi(\mu);Y_1)|^2<\infty, \; E |D_{r,r^\prime} h(\phi(\mu);Y_1)|<\infty \;\text{for}\; r,r^\prime =1,\ldots, s.
\end{equation}
\item [(A5)] (Locally uniform $L^1$-smoothness of the Hessian) Let $u_{r,r^\prime}(\epsilon;q)= \sup\{| D_{r,r^\prime} h(\theta;q) - D_{r,r^\prime} h(\phi (\mu);q)| : |\theta -\phi(\mu)|<\epsilon\}$.  Then
\begin{equation}
	E |u_{r,r^\prime}(\epsilon;Y_1)|\rightarrow 0 \;\text{as}\; \epsilon\rightarrow 0\;\text{for all}\; 1\leq r, r^\prime\leq s.
\end{equation}
\item [(A6)]  (Nonsingularity of the Hessian) The matrix   $\Lambda= [E D_{r,r^\prime} h(\phi(\mu);Y_1)]_{r,r^\prime=1,\ldots,s }$ is nonsingular.
\end{itemize}

\begin{remark}
\label{rem-2.1}
Observe that $Eh(x, Y_1)=F(\phi^{-1}(x))=ED_rh(x,Y_1)=D_rF(\phi^{-1}(x))$, $1\leq r\leq s$, $x\in U$. Also, $ED_rh(\phi (\mu), Y_1)=D_rF(\phi^{-1}(x))\mid_{x=\phi(\mu)}=0$, $1\leq r\leq s$, since $F(\phi^{-1}(x))$ attains a minimum at $x=\phi(\mu)$.
\end{remark}

\begin{theorem}
\label{th-clt}
Under assumptions (A1)-(A6) ,
\begin{equation}
\label{eq-maineq}
	n^{1/2}[ \phi(\mu_n)-\phi(\mu)] \xrightarrow{\mathcal{L}}  N(0, \Lambda^{-1}C \Lambda^{-1}), \;\text{as}\;  n\rightarrow \infty,
\end{equation}
where  $C $ is the covariance matrix of $\{D_rh(\phi(\mu);Y_1), r=1,\ldots,s\}$.
\end{theorem}

\begin{proof}
 The function $x\rightarrow F_n(\phi^{-1}x) = \dfrac{1}{n}\sum_{j=1}^n h(x,Y_j)$ on $U$ attains a minimum at  $\phi(\mu_n) \in U$ for all sufficiently large $n$ (almost surely).  For all such $n$ one therefore has the first order condition
\begin{equation}
\label{eq-firstorder}
\nabla \; F_n(\phi^{-1} \nu_n) = \dfrac{1}{n} \sum_{j=1}^n \nabla \; h(\nu_n,Y_j)  =0,
\end{equation}
where $\nu= \phi( \mu)$, $\nu_n= \phi( \mu_n)$ (column vectors in $U$). Here $\nabla$ is the gradient $(D_1,\ldots, D_r).$  A Taylor expansion yields
\begin{equation}
\label{eq-07}
	0 =  \dfrac{1}{n}\sum_{j=1}^n \nabla \;h(\nu_n,Y_j)  = \dfrac{1}{n}\sum_{j=1}^n \nabla \;h(\nu,Y_j)  +  \Lambda_n (\nu_n-\nu)
\end{equation}
where $\Lambda_n$  is the $s\times s$ matrix given by
\begin{equation}
\label{eq-08}
	\Lambda_n = \dfrac{1}{n}\sum_{j=1}^n [D_{r,r^\prime}h(\theta_{n,r,r^\prime},Y_j)]_{r,r^\prime=1,\ldots,s},
\end{equation}
and $\theta_{n, r, r^\prime}$ lies on the line segment joining $\nu_n$ and $\nu$.  We will show that
\begin{equation}
\label{eq-09}
	\Lambda_n \rightarrow \Lambda \;\text{in probability},\; \text{ as } n\rightarrow \infty.
\end{equation}
Fix  $r,r^\prime\in \{1,\ldots,s\}$.  For $\delta>0$, write $E u_{r,r^\prime} (\delta,Y_1) = \gamma(\delta)$. There exists $ n=n(\delta)$ such that  $ P(|\nu_n-\nu|>\delta)<\delta$ for $n>n(\delta).$  Now
\begin{align}
	E\big| [ \dfrac{1}{n}\sum_{j=1}^n D_{r,r^\prime}h(\nu_n,Y_j)- \dfrac{1}{n}\sum_{j=1}^n D_{r,r^\prime}h(\nu,Y_j)]\cdot 1_{[|\nu_n-\nu|\leq \delta]}\big|&\nonumber\leq E  \dfrac{1}{n}\sum_{j=1}^n u_{r,r^\prime} (\delta,Y_j)\\ \nonumber
 &= E u_{r,r^\prime} (\delta,Y_1) = \gamma(\delta) \rightarrow 0
\end{align}
 as  $\delta\rightarrow 0$.
Hence, by Chebyshev's inequality for first moments, for $n> n(\delta) $ one has for every $\epsilon>0$,
\begin{equation}
\label{eq-11}
	P(\big |  \dfrac{1}{n}\sum_{j=1}^n D_{r,r^\prime}h(\nu_n,Y_j)- \dfrac{1}{n}\sum_{j=1}^n D_{r,r^\prime}h(\nu,Y_j)\big|> \epsilon) \leq  \delta + \gamma(\delta)/\epsilon \rightarrow 0\;\text{ as}\;  \delta\rightarrow 0.
\end{equation}
This shows that
\begin{equation}
\label{eq-12}
	\big[ \dfrac{1}{n}\sum_{j=1}^n D_{r,r^\prime} h(\nu_n,Y_j)-  \dfrac{1}{n}\sum_{j=1}^n D_{r,r^\prime}h(\nu,Y_j)\big] \rightarrow 0;\text{ in probability as}\; n\rightarrow \infty.
\end{equation}
Next, by the strong law of large numbers,
\begin{equation}
\label{eq-13}
	 \dfrac{1}{n}\sum_{j=1}^n  D_{r,r^\prime}h(\nu,Y_j) \rightarrow  ED_{r,r^\prime}h(\nu,Y_1)\;\text{ almost surely,}\;\text{ as}\; n\rightarrow \infty.
\end{equation}
Since \eqref{eq-11} -- \eqref{eq-13} hold for all $r$,$r^\prime$, \eqref{eq-09} follows.  The set of symmetric $s\times s$ positive definite matrices is open in the set of all $s\times s$ symmetric matrices, so that \eqref{eq-09} implies that $\Lambda_n$ is nonsingular with probability going to 1 and  $\Lambda_n^{-1}\rightarrow \Lambda^{-1}$ in probability, as $n\rightarrow \infty$.  Note that $E\nabla h(\nu, Y_1)=0$ (see Remark \ref{rem-2.1}). Therefore, using (A4),  by the classical CLT and Slutsky's  Lemma, \eqref{eq-07} leads to
\begin{equation}
\label{eq-14}	\sqrt{n}(\nu_n-\nu) = \Lambda_n^{-1}[-(1/\sqrt{n})  \dfrac{1}{n}\sum_{j=1}^n \nabla\; h(\nu,Y_j)]  \xrightarrow{\mathcal{L}} N(0, \Lambda^{-1}C\Lambda^{-1}),
\end{equation}
as $ n\rightarrow \infty$.
\end{proof}

A preliminary version of Theorem \ref{th-clt} was presented in \cite{2013arXiv1306.5806B}.

\begin{corollary}[CLT for Intrinsic Means-I]
 \label{coro-1}  Let $(M,g)$ be a $d$-dimensional complete Riemannian manifold with metric tensor $g$ and geodesic distance $\rho_g$. Suppose  $Q$ is a probability measure on $M $ with intrinsic mean $\mu_I$, and that $Q$ assigns zero mass to a neighborhood, however small, of the \emph{cut locus} of $\mu_I$.  Let $\phi= \exp \mu_I^{-1}$ be the inverse exponential, or $\log$-, function at $\mu_I$ defined on a neighborhood $G$ of $\mu = \mu_I$ onto its image $U$ in the tangent space $T_{\mu_I} (M)$. Assume that the assumptions (A4)-(A6) hold. Then, with $s=d$, the CLT  \eqref{eq-maineq} holds for the intrinsic sample mean $\mu_n = \mu_{n,I}$, say.
\end{corollary}

\begin{remark}
Corollary \ref{coro-1} improves the CLT for the intrinsic mean due to \cite{rabivic05}, and also Theorem 2.3, and Theorem 5.3 in \cite{rabibook}.  
\end{remark}

  For the case of the \emph{extrinsic mean}, let $M$ be a $d$-dimensional differentiable manifold, and $J: M\rightarrow  E^N$ an embedding of $M$ into an $N$-dimensional Euclidean space. Assume that $J(M)$ is closed in $E^N$, which is always the case, in particular,  if $M$ is compact. The extrinsic distance $\rho_{E,J}$ on $M $ is defined as $\rho_{E,J}(p,q) = \|J(p)-J(q)\|$ for $p,q \in M$, where  $\|\cdot \|$ denotes the Euclidean norm of $E^N$. The image $\mu$ in $J(M)$ of the \emph{extrinsic mean} $\mu_{E,J }$ is then given by $\mu= P(m)$, where $m$ is the usual mean of $ Q\circ J^{-1}$ thought of as a probability on the Euclidean space $E^N$, and $P$ is the orthogonal projection defined on an $N$-dimensional neighborhood $V$ of $m$ into $J(M)$ minimizing the Euclidean distance between $p\in V$ and $J(M)$.  If the projection $P$ is unique on $V $ then the projection $\mu_n= P(m_n)$ of the Euclidean mean $m_n = \sum_{j=1}^n J(Y_j)/n$ on $J(M)$ is, with probability tending to one as $n\rightarrow \infty$, unique and lies in an open neighborhood $G$ of $\mu =P(m)$ in $J(M)$.  Theorem \ref{th-clt} immediately implies the following result of \cite{rabivic03} (Also see \cite{rabibook}, Proposition 4.3).

\begin{corollary}[CLT for Extrinsic Means on a Manifold]
  \label{coro-2} Assume that $P$ is uniquely defined in a neighborhood of the $N$-dimensional Euclidean mean $m$ of $Q\circ J^{-1}$.  Let $\phi$ be a diffeomorphism on a neighborhood $G$  of $\mu= P(m)$ in $J(M)$ onto an open set $ U$ in $\mathbb R^d$. Assume (A1), (A4)-(A6). Then, using the notation of \eqref{eq-maineq},
  \begin{equation*}
  \sqrt{n} \left[\phi(\mu_n)-\phi(\mu) \right]= \sqrt{n}\left[\phi(P(m_n))- \phi( P(m))\right] \xrightarrow{\mathcal{L}} N(0, \Lambda^{-1}C\Lambda^{-1}), \;\text{as}\; n\rightarrow \infty.
  \end{equation*}
\end{corollary}

\begin{remark} In Corollary \ref{coro-2}, one may, in particular, choose $(U,\phi)$ to be a coordinate neighborhood of $\mu=P(m)$ in $J(M)$. In \cite{rabivic03}, however, $\phi$ is chosen to be the linear orthogonal projection on $G$ into the tangent space $T_{\mu} J(M)$. 
\end{remark}

%


\begin{remark}
In the case $S=M$ is a Riemannian manifold and ($G=M$), the dispersion matrix in Theorem \ref{th-clt} ( and  Theorem \ref{th-2.4}  in the next section) is related to the sectional curvature of $M$. For $M$ with constant curvature such as $S^d$ one may express this matrix explicitly  (See \cite{absrabi1}). Recently, \cite{kendall2011} has extended this result to the important case of planar shape space $\Sigma_2^k$ and, more generally to manifolds with constant holomorphic curvature.
\end{remark}

We now turn to applications of Theorem \ref{th-clt} to the so-called \emph{stratified spaces} $S$ which are made up of several subspaces of different dimensions. 
In particular, we next consider an example where $S$ is a \emph{space of non-positive curvature} (NPC), which is not in general a differentiable manifold, but has a metric with properties of a geodesic distance (namely, minimum length of curves between points) and which is also somewhat analogous to differentiable manifolds of non-positive curvature. These spaces were originally studied by A.D. Alexandrov and developed further by Yu. G. Reshetnyak and M. Gromov (See  \cite{stu03} for a detailed treatment).  Unlike differentiable manifolds of positive curvature where uniqueness of the intrinsic mean is known only under very restrictive conditions (See \cite{karcher}, \cite{kenws} and \cite{afsari11}), on an NPC space the Fr\'echet mean is always unique, if the Fr\'echet function \eqref{eq-frechet} is finite \cite{stu03}.

We will consider a stratified NPC space $S$ which is the union of a finite number of disjoint  sets $U_k$ each of which in its relative topology in $S $ is homeomorphic to an open subset of $\mathbb R^s$, including possibly the degenerate case $s=0$, $\mathbb R^0$ being a singleton.

The results described below originated in a SAMSI working group (\url{http://www.samsi.info/working-groups/data-analysis-sample-spaces-manifold-stratification}), and are further developed in \cite{rabietall11}, \cite{openbook13}. Also see \cite{basrak10} , \cite{Osborne2013} and \cite{EJP3887}.

 Let $Q $ be a probability measure on $S$. We define the \emph{Wasserstein distance} $d_W$ on the space $\mathcal{P}(S)$ of probability measures on the Borel sigma-field of $S$ as
\begin{equation}
	   d^2_W(Q_1,Q_2) = \inf\{ E\rho^2(\boldsymbol X,\boldsymbol Y): \mathcal L(\boldsymbol X)= Q_1, \mathcal L(\boldsymbol Y) = Q_2\} ,
\end{equation}
where $\mathcal L(\boldsymbol Z)$ denotes the law, or distribution, of $\boldsymbol Z$. That is, the infimum on the right is over the set of all (joint) distributions of $(\boldsymbol X,\boldsymbol Y)$ (in $\mathcal{P} (S\times S))$ with marginals $Q_1$ and $Q_2$.  For considering  finite Fr\'echet functions the appropriate space of probabilities that we consider below is $\{\widetilde{Q}\in \mathcal{P} (S): \text{Fr\'echet function of}\; \widetilde{Q} \;\text{is finite}\}$, endowed with the Wasserstein distance.

On a stratified space $S$, we say that the \emph{Fr\'echet mean} $\mu$ of $Q$ is \emph{sticky} on $U_k$, if there exists a  Wasserstein neighborhood of $Q$ such that for every $\widetilde{Q}$ in this neighborhood the Fr\'echet mean of $\widetilde{Q}$ lies in the same stratum $U_k$.

As an immediate consequence of Theorem \ref{th-clt}, we get the following result.

\begin{proposition} Suppose the Fr\'echet mean $\mu$ of $Q$ on a stratified NPC space $S$ is sticky on a stratum $U_k$ which is not degenerate. Then, with $G=U_k$,  the CLT in Theorem \ref{th-clt} holds under the given assumptions \eqref{eq-asa2} and (A4)-(A6). In the degenerate case, i.e., $U_k = \{\mu\}$, the sample Fr\'echet mean $\mu_n$ equals $\mu$ for all sufficiently large $n$, almost surely.
\end{proposition}

\begin{example}[\emph{Open Book}]

Let $S= \left(\cup_{k=1,\dots,K }  H_k\right)\cup S_0$ where $H_k : = \{k\}\times H$, $H= \mathbb R^D \times [0, \infty)$, $S_0 = \{0\}\times \mathbb R^D$, with the boundary point $(k; 0,x^1,\ldots,x^D)$ of $H_k$ identified with the point $(0,x^1,\ldots,x^D)$ of $S_0$ for all $k$.  That is, $S$ is the union of $K$ copies of the half space $H$ glued together at the common border or \emph{spine} $S_0 = \{0\}\times \mathbb R^D$. We express $S$ as the disjoint union $S = (\cup_{ k=1,\ldots,K}  S_k )\cup S_0$, where the $k$-th \emph{leaf} is $S_k = \{(k;x^0,x^1,\ldots,x^D)\}$ with $x^0 \in (0, \infty)$, $x^j \in \mathbb R$ for  $j=1,\dots,D$. For a point $\boldsymbol x = (x^0, x^1,\ldots,x^D)  \in H$ we define its \emph{reflection across the spine}  $S_0$ as $R\boldsymbol x = (-x^0, x^1,\ldots,x^D)$. Using $\| \cdot\|$ for the Euclidean norm, the distance $\rho$ on $S$ is then defined by
\begin{align}
	    \rho((k;\boldsymbol x), (k;\boldsymbol y))& = \|\boldsymbol x-\boldsymbol y\| \; \forall \boldsymbol x, \boldsymbol y \in H= \mathbb{R}^D \times [0, \infty), k=1,\ldots,K;\\ \nonumber
\rho((k;\boldsymbol x), (k';\boldsymbol y)) &= \|\boldsymbol x-R\boldsymbol y\|= \| R\boldsymbol x -\boldsymbol y\| \; \forall \boldsymbol x, \boldsymbol y \in H= \mathbb R^D \times [0, \infty), \;\text{if}\; k\neq k'.
\end{align}
Note that while the zero-th coordinate  $x^0$ of $\boldsymbol x$ is nonnegative, that of $R\boldsymbol x$ is $-x^0$ and is negative or zero, so that if $k \neq k'$
\begin{equation}
	\rho^2\left((k;x^0, x^1,\ldots,x^D), (k';y^0,y^1,\ldots,y^D)\right) = (x^0+y^0)^2 + \|(x^1,\ldots,x^D)- (y^1,\ldots,y^D)\|^2 .
\end{equation}
\end{example}

   We now provide an exposition of a characterization of sticky Fr\'echet means on open books due to  \cite{openbook13}: ``Sticky central limit theorems on open books", with slightly different notations and terminology.  Assume that $w_k= \mu(S_k) > 0 $ for all $k=1,\ldots,K.$  Define the following  \emph{$k$-th folding map} $f_k$ on $S$ into $\mathbb{R}^{D+1}$ as
\begin{equation}
	f_k((k:\boldsymbol x)) =  \boldsymbol x,   f_k((k':\boldsymbol x)) = R\boldsymbol x\;  \text{if}\; k'\neq k \; (k=1,\dots,K),
\end{equation}
and denote by $m_k$ the usual (one-dimensional) mean of the zero-th coordinate of  $f_k$:
\begin{equation}
\label{eq-mk}
	m_k = \int z^0 (Q\circ f_k^{-1})(d\boldsymbol z) = -(1/2)\left[\partial /\partial x^0\int_{\mathbb R^{d+1} }\|\boldsymbol z-\boldsymbol x\|^2 (Q\circ f_k^{-1})(d\boldsymbol z) \right]_{x^0=0}.
\end{equation}

Let $\widetilde{Q}$ be the distribution induced by $Q $ on $H $ under the projection $\pi$ on $S$ into $H$ defined by $\pi(k;\boldsymbol x)=\boldsymbol x$  (and $\pi(\boldsymbol x)=\boldsymbol x $ on $S_0$).  Let $Q_k$ be the measure $\widetilde{Q}$ restricted to $\pi(S_k)$. Note that $Q_k = Q\circ f_{k}^{-1}$ restricted to $S_k$.  Also, let $Q_0$ be the restriction of $Q$ (or $\widetilde{Q}$) to $S_0.$
In view of the additive nature of $\rho^2$, the minimization of the Fr\'echet function is achieved separately for the zero-th coordinate $x^0$ of $\boldsymbol x$ along with the leaf on which it lies, and the remaining D coordinates $(x^1,\ldots,x^D)$. The last $D$ coordinates of the Fr\'echet mean on $S$ is simply the mean $\mu_{1D}$, say,  of $(x^1,\ldots,x^D)$ under $\widetilde{Q}$.  The position of the Fr\'echet mean $\mu$, or whether it is sticky on the spine $S_0$ or to some other stratum, is determined by $m_k$ ($k=1,\ldots,K)$.   Since the integral on the right side of \eqref{eq-mk} is the Fr\'echet function of $Q$ evaluated on the leaf $S_k$ at the spine, it follows from \eqref{eq-mk} that \emph{if $m_k>0$, then, for a while,  the Fr\'echet function is strictly decreasing on $S_k$ along the zero-th coordinate as it moves away from the spine $S_0$.}  On the other hand, if $m_k > 0$ then $m_{k'} <0$ for all $ k'\neq k$. For this note that $m_k =  \int_{H} \boldsymbol z^0 Q_k(d\boldsymbol z) - \sum_{1\leq k'\neq k} \int_{H} \boldsymbol z^0 Q_{k'}(d\boldsymbol z)$.  Comparing this with the corresponding expression for $m_{k'}$, we see that $m_{k'}  \leq  \int_H \boldsymbol z^0 Q_{k'}(d\boldsymbol z) - \int_{H }\boldsymbol z^0 Q_k(d\boldsymbol z) <0$, since  $m_k >0.$  Hence \emph{the Fr\'echet function is strictly increasing on $S_{k'}$ for all $k'\neq k$ along the zero-th coordinate as it increases, i.e., as the point moves away from the spine $S_0$.} It follows that $\mu \in S_k$.  Also, if $m_k>0$ then there exists a neighborhood of $Q$ in the Wasserstein distance on which $m_k>0$.  That is, if $m_k>0$ for some $k$, then $\mu$ is sticky on the stratum $S_k$ , and Theorem \ref{th-clt} applies with $s= D+1 =d.$  It is clear that the Fr\'echet mean in this case is $\mu = (k; m_k, \mu_{1D})$, and the asymptotic distribution of $\pi(\mu_n)$ is Normal with mean $(m_k, \mu_{1D} )$, and  covariance matrix $n^{-1}\Sigma$, where $\Sigma$ is the $d\times d$  covariance matrix of $Q\circ f_{k}^{-1}$, which follows from the classical multivariate CLT for i.i.d. summands with common distribution $Q\circ f_{k}^{-1}$.  The above argument also shows that if $m_k <0$ for all $k=1,\ldots,K,$  then $\mu$  belongs to $S_0$, and it is \emph{sticky on the spine $S_0$}, so that Theorem \ref{th-clt} applies with $s= D$. In this case $\mu= (0, \mu_{1D})$ and, with probability tending to one as $n\rightarrow \infty$, $\mu_n$ lies in $S_0$, with its zero-th coordinate as 0, and its remaining $D$ coordinates comprising the mean of $n$ i.i.d. vectors with the common distribution that of  $(X_1,\ldots,X_D)$ under $Q$. Thus, again, by the classical multivariate CLT for i.i.d. summands, the asymptotic distribution of $\mu_n = \pi(\mu_n)$ on $S_0$ is Normal $N((0, \mu_{1D}), n^{-1}\Sigma_0)$. Note that $\Sigma_0$ is the same as the $D\times D$ upper sub-matrix of $\Sigma$.

 To complete the picture consider the case  $m_k =0$ for some $k$. Then once again $m_{k'}<0 $ for all $k'\neq k$, and the minimum of the Fr\'echet function occurs on $S_0\cup S_k = \bar{S}_k$. Let $m_{k,n}$ be the sample mean of the zero-th coordinate under $Q\circ f_{k}^{-1}$.  Since the set $\{Q': m_{k'}<0$ for all $k'\neq k\}$ is open in the Wasserstein distance (in the set of probabilities $\{Q':\text{ Fr\'echet function of $Q'$ is finite}\}$), if  $m_{k,n} \leq 0$ then the sample Fr\'echet mean $\mu_n$ lies in $S_0.$ If $m_{k,n} > 0$, then $\mu_n$ lies in $S_k$.   Since $E(m_{k,n})=m_k= 0$, it follows by the classical CLT that the asymptotic distribution of $\mu_n$ is, with probability $\frac{1}{2}$, $N((0, \mu_{1D}), n^{-1}\Sigma_0 )$ on $S_0$ and, with probability $\frac{1}{2}$,  it has the asymptotic  distribution on $S_k$ of its numerical coordinates as the conditional distribution of    $(X^0, X^1,\ldots , X^D)$, given $X^0>0$, where $(X^0,X^1,\ldots ,X^D)$ has the distribution $N((0, \mu_{1D}), n^{-1}\Sigma)$.

We refer to other examples of stratified spaces such as considered in \cite{ref-Felsen} and \cite{barden13}, where also Theorem \ref{th-clt} applies. These may be thought of as toy models for the study of phylogenetic trees pioneered by S. Holmes and her collaborators (see, e.g., \cite{Billera2011},  \cite{holmes05}).

\section{A CLT for the intrinsic mean}
\label{sec-3}

We begin with the circle  $S^1$. Under the assumption of a continuous density $f$ of $Q$ on $S^1$, a necessary and sufficient condition for the existence of a unique minimizer of the intrinsic Fr\'echet function on the circle $S^1$ was given in the  manuscript \cite{rabinotes}, showing, in particular, the twice continuous differentiability of the intrinsic Fr\'echet function. It is further shown there that the Fr\'echet function is convex at $p\in S^1$  if $f(-p) <1/2\pi$, concave if $f(-p) >1/2\pi$ . This work is mentioned in \cite{huckman_circle}, p. 182, and also appears in \cite{rabibook}, pp. 73-75, 31-33. Under a continuity assumption, a direct proof of the CLT of the Fr\'echet mean is given in \cite{mckilliam2012}, and extended further in \cite{huckman_circle} when the continuity assumption does not hold.

\begin{proposition}
\label{prop-3.1}
On $S^d$  the Fr\'echet function is twice continuously differentiable if $Q$ has a twice continuously differentiable density $f$. 
\end{proposition}

\begin{proof}
For this one expresses the Fr\'echet function as  $F(p) = \int_{D_\pi}\|v\|^2f(\exp_pv)m(dv)$ with a natural identification with the disc $D_{\pi} =\{v:0\leq \|v\|<\pi\} $ $(\subset \R^d)$ of the image of $S^d \backslash \{-p\} $ in $T_pS^d $ under the map $\log_p$, and $m(dv)$  denoting the measure induced on $T_pS^d$  from the volume measure on $S^d$  by the map $log_p$, thought of as a measure on $D_\pi$ by corresponding identifications for all $p$. 
\end{proof}

\begin{remark}

Since the squared intrinsic distance $\rho_g^2(p,q)$  is smooth in $p$  for $q$ outside any neighborhood of $\{-p\}$,  it is probably enough to assume that $f$ has continuous derivatives of order one, or even that  $f$ is continuous.  Also, we expect Proposition \ref{prop-3.1} and its proof to carry over to more general Riemannian manifolds such as those which are \emph{homogeneous} (\cite{docarmo}, p.154).
\end{remark}

On a general complete connected $d$-dimensional Riemannian manifold $(M, g)$,  the \emph{cut point} of a point $p$ along a geodesic $\gamma(t)$, $t\geq 0$ ($\gamma(0)=p$) is $\gamma(t_0)$, where $t_0=\sup\{t\geq 0: \gamma(u), 0\leq u\leq t,\;\text{ is the unique distance minimizing }\\ \text{segment of}  \;\gamma\;\text{  between}\;   p\;\text{  and}\; \gamma(t)\}$.  The set of all cut points of $p$ along geodesics is called the \emph{cut locus} of $p$ and denoted $C(p)$ (\cite{docarmo}, p. 207). Suppose the intrinsic mean $\mu_I$ of a probability measure $Q$ on $M$ exists. Take $\mu=\mu_I$,   $\phi(p) = \log_{\mu}(p)$ defined on $M\backslash C(\mu)$.  Then $\phi^{-1}(x) = \exp_{\mu}(x)$ and $x\rightarrow  h(x,q) $ is twice continuously differentiable on  $J((M\backslash C(\mu))\backslash C(q))$.  Observe that $p\in C(q)$ if and only if $q\in C(p)$ (\cite{docarmo}, p. 271). By a slight abuse of notation, we will denote by $C(U)$ the set of cut loci of all points in a set $U\subset M$. Let $B(\mu;  \epsilon)$ denote the geodesic ball with center $\mu$ and radius $\epsilon$. Then $\phi(B(\mu; \epsilon))$ is the ball in $T_{\mu}M$ with center $\nu=\phi(\mu)=0$ and radius $\epsilon$. We then have the following result. 



\begin{theorem}[CLT for Intrinsic  Means-II]
\label{th-2.4}
Suppose that $Q$ has an intrinsic mean  $\mu$, and that $Q$ is absolutely continuous in a neighborhood $W$ of the cut locus of $\mu$ with a continuous density there with respect to the volume measure. Assume also that (i) $Q(C(B(\mu;\epsilon))) = O(\epsilon^{d-c})$, $\epsilon\rightarrow 0$,  for some $c$, $0\leq c< d$, (ii) on some neighborhood $V$ of $\nu = \phi(\mu) =0$ the function $\theta\rightarrow F\left(\phi^{-1}(\theta)\right)$  is twice continuously differentiable with a nonsingular Hessian $\Lambda(\theta)$, and (iii) (A4) holds with $\phi(\mu)$ replaced by $\theta $, $\forall$$\theta\in V$. Then, if $d>c+2$, one has the CLT \eqref{eq-maineq}  for the sample intrinsic mean $\mu_n$.
\end{theorem}

\begin{proof}
 Without loss of generality we take the neighborhood $V $ of $\nu=0$ sufficiently small such that $C(\phi^{-1}(V))\subset W$. Then $Z_n(\theta):= n^{-1}\sum_{1\leq j\leq n}\text{grad}\;h(\theta,Y_j)$ is well defined for $Y_j\not\in C(\phi^{-1}\theta)$, $j=1,\ldots, n$, that is,   with  probability one,  provided $\theta\in V$, since  $Q(C(\phi^{-1}\theta))=0$. By the classical CLT, $Z_n(0):= n^{-1}\sum_{1\leq j\leq n}\text{grad}\;h(0,Y_j) $ is of the order $O_p(n^{-1/2}).$  Let $B_n$  be the ball in $T_{\mu}M$ with center $\nu= \phi(\mu)=0$ and radius $n^{-1/2}\log n$. By hypothesis, the probability that  $Y_j\in C(\phi^{-1}(B_n))$ is $O((n^{-1/2}\log n)^{d-c} )$.  For $\phi^{-1}(B_n)$ is the geodesic ball $B(\mu; n^{-1/2}\log n)$, hence the probability that the set $ \{Y_j:j=1,\ldots, n\}$  intersects $C(\phi^{-1}(B_n))$  is $O(n(n^{-1/2}\log n)^{d-c} ) = o(1)$ if $d>c+2$. Hence with probability converging to 1, one may use a Taylor expansion of $Z_n(\theta)$ in $B_n$, 
 \begin{align}
Z_n(\theta) = Z_n(\nu) + \Lambda_n(\theta)(\theta-\nu),\;\; (\theta\in B_n), \; (\nu=0),     
\end{align}                                                                       
where $\Lambda_n(\theta)$ is the $d\times d$ matrix whose $(r,r') $ element is  $n^{-1}\sum_{1\leq j\leq n}D_{r,r'}h(\theta(n;r,r', Y_j), Y_j)$  with $\theta(n;r,r',Y_j)$ lying on the line segment joining $\theta$ and $\nu=0$. 
By hypothesis (ii), with probability converging to one as $n\rightarrow \infty$, $\Lambda_n(\theta)$  is nonsingular for all large $n$ ($\theta\in B_n$) since its difference (in norm) from the Hessian $\Lambda(\theta)$  goes to zero as $n\rightarrow \infty$, by the strong law of large numbers. Now, with probability going to 1, the function $\theta\rightarrow H_n(\theta) =0 -\Lambda_n(\theta)^{-1}Z_n(\nu)$ maps  $\bar{B}_n$   into itself, where  $\bar{B}_n$  is the closure of $B_n$. For this argument recall that $Z_n(0) = O_p(n^{-1/2})$ by the classical CLT. By the \emph{Brouwer fixed point theorem} (\cite{milnor1997topology}), $H_n(\theta) $ has a fixed point. Letting $\nu_n$ denote a measurable selection from the set of fixed points in  $\bar{B}_n$ , it follows that, with probability going to 1, $\nu_n$ converges to $\nu$ and satisfies the first order equation \eqref{eq-firstorder}. Hence one may take  $\nu_n$ as the sample intrinsic mean (Note that the Fr\'echet function is strictly convex in a neighborhood of $\nu$). The CLT now follows as in the last line of the proof of Theorem  \ref{th-clt}.  

\end{proof}

\begin{remark}
For $d\leq c+2$ the condition (i) in Theorem \ref{th-2.4} does not imply that the probability  the set  $\{Y_1,\ldots, Y_n\}$ intersects  $C(\phi^{-1}(B_n))$   goes to zero. Intuitively one may think that the cut locus of the image under $\phi^{-1}$ of a small neighborhood of the  \emph{random line} joining $\nu_n$ and 0 intersecting  $\{Y_1,\ldots, Y_n\}$ is negligible; but we do not know how to justify this intuition or that it is even true. 

\end{remark}

\begin{corollary} 
\label{coro-2.5}
Suppose $Q$ on $M= S^d$ ($d>2$) has an intrinsic mean $\mu$ and is absolutely continuous on a neighborhood $W$ of $C(\mu)$ with a continuous density on $W$. Suppose that the hypotheses (ii), (iii) of Theorem \ref{th-2.4} hold. Then the CLT for the sample intrinsic mean holds.
\end{corollary}

\begin{proof} 
It is enough to note that the hypothesis (i) in Theorem \ref{th-2.4} holds. Note that in the present case $C(\mu)=\{-\mu\}$ and $C(\phi^{-1}(B_n))$ is the set $-\phi^{-1}(B_n)=\{-B(\mu; n^{-1/2}\log n)\}=B(-\mu; n^{-1/2}\log n)$. The probability that $\{Y_1,\ldots, Y_n\}$ intersects this last set is $O(n(n^{-1/2}\log n)^d ) $, since the density of $Q$ on a small compact neighborhood of $C(\mu)$ is bounded.

\end{proof}

\begin{remark}
As mentioned at the beginning of this section, $F$  is twice continuously differentiable if $Q$ has a twice continuously differentiable density. We expect that the proof can be extended to the case where $Q$ has a smooth density only in a neighborhood of $C(\mu).$  In the case of $S^1$ this is known under the assumption of just continuity of the density at $\mu$ (See \cite{huckman_circle} or the proof in \cite{rabibook} or \cite{rabinotes}).  It is for this reason we have not assumed in Theorem \ref{th-2.4} and Corollary \ref{coro-2.5} that $Q$ has a smooth density, although the Fr\'echet function is assumed to be twice continuously differentiable in a neighborhood $C(\mu)$. 
\end{remark}

\begin{remark}

Although it is curious that the proof of Theorem \ref{th-2.4}  does not hold for $d=2$, the authors expect that a proof of Corollary \ref{coro-2.5}  for the case $d=2$ may be given using polar coordinates.  For the moment the CLT for $S^2$  is derived only under the support  restriction of Corollary \ref{coro-1}. 
\end{remark}

\begin{remark}
\label{remark-3.8}
 Suppose $\mathcal G$ is a Lie group of isometries on $S^d$, $d>2$. Then the projection $\pi: S^d\rightarrow S^d/\mathcal G$  is a \emph{Riemannian submersion} on $S^d$ onto its quotient space $M= S^d/\mathcal G$ (\cite{gl} , pp. 63-65, 97-99). Let $Q$ be a probability measure on $S^d$ with a twice continuously differentiable density and a Karcher or  intrinsic mean $\mu$. Let $\tilde \mu$ be the projection of $\mu$. Then, in local coordinates, the differential of the Fr\'echet function on $M$ vanishes at $\tilde \mu$, because $\pi$ is smooth and the differential of the Fr\'echet function on $S^d$ vanishes at $\mu$. If $\tilde \mu$ is a Karcher or intrinsic mean of $\tilde Q$, then the delta method provides a CLT for the corresponding sample Fr\'echet mean $\tilde \mu_n$  in local coordinates. If $\tilde \mu$ is just a local minimum, one can still use the CLT for two sample problems (See \cite{absrabi1, rabibook}).  One may also explore the opposite route for a probability $\tilde Q$ on $M$ with a density and a unique intrinsic/Karcher mean $\tilde \mu$  and a probability $Q$, among a fairly large family of distributions with smooth densities on $S^d$  whose projection on $M$  is $\tilde Q$, such that $Q$ satisfies the hypothesis of Corollary \ref{coro-2.5} with  $\pi(\mu) =\tilde \mu$ . One may then apply the CLT on $S^d$  to derive one on $S^d/\mathcal G$. As an example consider the antipodal map $g(p)=-p$, and $\mathcal G= \{g, \text{identity}\}$. Let $\tilde Q$ be a probability on $M= S^d/\mathcal G = \mathbb RP^d$ (the real projective space) thought of as a probability on the upper hemisphere vanishing smoothly at the boundary, and with a unique intrinsic mean $\tilde \mu=\{\mu,-\mu\}$, where $\mu$ is the Karcher mean of $Q $ (restricted to the hemisphere). This opens a way for CLT's on Kendall's shape spaces as well.

\end{remark}

\begin{remark}
Instead of defining the Fr\'echet mean with restricted to the squared distance $\rho^2$, one may define it with respect to $\rho^{\alpha}$, $\alpha\geq 1$, in \eqref{eq-frechetf}, and derive Theorems \ref{th-clt}, \ref{th-2.4}, if the assumptions hold with respect to $\rho^{\alpha}$ in place $\rho^2$. Note that Proposition \ref{prop-3.1} extends easily to this case. 
\end{remark}

\begin{remark}
As indicated in Remark \ref{remark-3.8}, one of the significances of a CLT on $S^d$ is that it may provide a  route to intrinsic CLTs on $S^d/\mathcal G$, the space of orbits under a \emph{Lie group $\mathcal G$ of isometries} of $S^d$.  Such spaces include the so-called \emph{axial spaces} (or \emph{real projective spaces} $\mathbb R P^d$), and Kendall type  shape spaces which are important in \emph{shape-based image analysis}.  For the latter spaces $S^d$ is the so-called \emph{preshape sphere} (see, e.g., \cite{rabibook}, p.82). Observe  that  the hypothesis (i) of Theorem \ref{th-2.4} may not hold in all such quotient spaces. For example, on $\mathbb R P^d$ one only has the order $O(\epsilon)$ in hypothesis (i) in Theorem \ref{th-2.4}, since the cut locus of the a point in $\mathbb{R} P^d$ is isomorphic to $\mathbb{R} P^{d-1}$.   For Kendall's planar shape space, identified as the \emph{complex projective space} $\mathbb{C}P^{k-2}$, of dimension $d=2k-4$, the volume measure of $C\left(B(\mu;\epsilon)\right)$ is $O(\epsilon^{2})$,  since the cut locus of a point of $\mathbb{C}P^{k-2}$ is isomorphic to $\mathbb{C}P^{k-3}$. For these facts refer to \cite{gl}, Section 2.114, pp. 102, 103. 
\end{remark}

\section{Real data examples}
\label{sec-data}

\subsection{Kendall's planar shape space (Corpus Callosum shapes of normal and ADHD children)}

We consider a planar  shape data set,  which involve measurements of a group typically developing children and a group of children suffering the ADHD (Attention deficit hyperactivity disorder).  ADHD  is one of the most common psychiatric  disorders for children that can continue through adolescence and adulthood. Symptoms include difficulty staying focused and paying attention, difficulty controlling behavior, and hyperactivity (over-activity). ADHD  in general has three subtypes: (1) ADHD hyperactive-impulsive (2) ADHD-inattentive; (3) Combined hyperactive-impulsive and inattentive  (ADHD-combined) \cite{ADHDtype}.  ADHD-200 Dataset (\url{http://fcon_1000.projects.nitrc.org/indi/adhd200/}) is a data set that  record both anatomical and resting-state functional MRI data of 776 labeled subjects across 8 independent imaging sites, 491 of which were obtained from typically developing individuals and 285 in children and adolescents with ADHD (ages: 7-21 years old).  
 The  Corpus Callosum shape data are extracted  using the CCSeg package, which contains 50 landmarks., with 50 landmarks on the contour of the Corpus Callosum of each subject (see \cite{hongtu15}). 
 After quality control,  647 CC shape data out of 776 subjects were obtained, which included 404 ($n_1$) typically developing children, 150 ($n_2$) diagnosed with ADHD-Combined, 8 ($n_3$) diagnosed with ADHD-Hyperactive-Impulsive,  and 85 ($n_4$) diagnosed with ADHD-Inattentive. Therefore, the data lie in the space $\Sigma_2^{50}$, which has a high dimension of $2\times 50-4=96$.   To provide a better picture of the data, we give displays of the landmark data  by making the scatter plots of the landmarks selected from the contours of the CC midsections, for the  243 young individuals diagnosed with ADHD. See Figure \ref{raw}.

    We carry out  \emph{extrinsic two sample tests} based on Corollary \ref{coro-2}  between the group of typically developing children and the group of children diagnosed with ADHD-Combined, and also between the group of typically developing children  and ADHD-Inattentive children. We construct test statistics that  base on the asymptotic distribution of the extrinsic mean for the planar shapes.
    \begin{figure}[ht]
\caption{Raw landmarks from the contour of the Corpus Callosum for 243 ADHD child}
\label{raw}
\includegraphics[width=8cm]{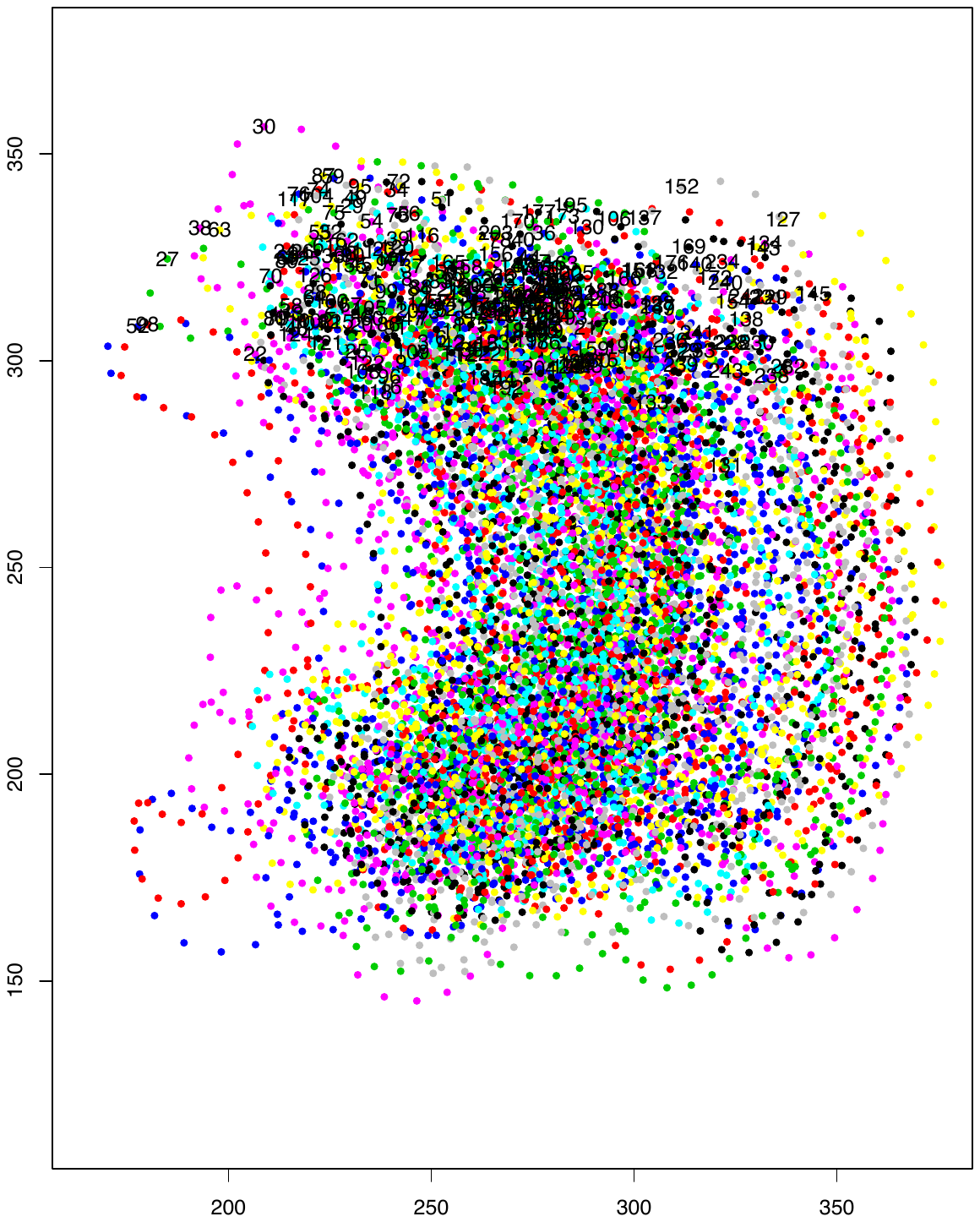}
\end{figure}
 

The $p$-value for the  two-sample test between the group of typically developing children and the group of children diagnosed with ADHD-Combined is $5.1988\times 10^{-11}$, which is based on the asymptotic chi-squared distribution given in Corollary \ref{coro-2}. The $p$-value for the test between the group of typically developing children and the group  ADHD-Inattentive children is smaller than $10^{-50}$.  It has been suggested the small $p$-values may result from the high dimension of the data. An alternative approach may perhaps be based on neighborhood testing in the context of Hilbert manifolds in which the shape contour is treated as an infinite-dimensional object \cite{Ellingson2013, hilbert-mani, vicbook}.

The planar shape data and the codes used for computing the $p$-values can be found in \url{http://www.stat.duke.edu/~ll162/research/planar.zip}.

\subsection{Positive definite matrices with application to diffusion tensor imaging}

  We consider $\sym^{+}(p)$, the space of $p\times p$ positive definite matrices.
Let $A\in \sym^{+}(p)$ which  follows a distribution $Q$.  The Euclidean metric of $A$ is given by $\|A\|^2=\trace(A)^2$.  Since $\sym^{+}(p)$ is an open convex subset of $\sym(p)$, the space of all $p\times p$ symmetric matrices, the mean  of $Q$  with respect to the Euclidean distance is given by the  Euclidean mean
\begin{equation}
\mu_E=\int A Q(dA).
\end{equation}

 Another  metric for $\sym^{+}(p)$  is the \emph{$\log$-Euclidean metric} \cite{Arsigny06}. Let $J\equiv \log:  \sym^{+}(p)\rightarrow \sym(p)$ be the inverse of the \emph{exponential map} $B\rightarrow e^B$,  $\sym(p)\rightarrow \sym^{+}(p)$, which is the matrix exponential of $B$. $J$ is a diffeomorphism. The $\log$ Euclidean distance is given by
\begin{equation}
\rho_{LE}(A_1,A_2)=\| \log(A_1)-\log(A_2)\|.
\end{equation}
Note that $J$ is an embedding on  $\sym^{+}(p)$ onto $\sym (p)$ and, in fact, it is an equivariant embedding under the group action of $\text{GL}(p, \R)$ , the  general linear group of $p\times p$ non-singular matrices. The extrinsic mean of $Q$ under $J$ is given by
\begin{equation}
\mu_{E,J}=\exp(\int (\log (A))Q(dA)).
\end{equation}

Also, this is the \emph{intrinsic mean} of $Q$ under the bi-invariant metric of $\sym^{+}(p)$ as a Lie group under multiplication: $A_1\circ A_2=\exp(\log(A_1)+\log(A_2))$.
Since it is also the metric inherited from the vector space $\sym(p)$, $\sym^+(p)$ has zero sectional curvature. Another commonly used metric tensor on $\sym^+(p)$ is the \emph{affine metric}: $\langle\langle Y, Z\rangle\rangle_A=\trace(A^{-1}YA^{-1}Z)$ $\forall Y, Z\in\sym(p).$ It is known that, with this metric, $\sym^+(p)$ has non-positive curvature \cite{affine-invariate}. We do not use this in our DTI data example, because it is computation intensive and yields results are often indistinguishable from those using the log-Euclidean metric \cite{2014arXiv1407.6383S}.


 Theorem \ref{th-clt} applies to  sample Fr\'echet means under both the Euclidean and $\log$-Euclidean distances.  Let $X_1,\ldots, X_{n_1}$ be an i.i.d sample from $Q_1$ on  $\sym^{+}(p)$ and $Y_1,\ldots, Y_{n_2}$ be an i.i.d sample from  distribution $Q_2$ on  $\sym^{+}(p)$, with $\bar X$ and $\bar Y$ their corresponding sample means.  Consider the case $p=3$, $\bar X$ and $\bar Y$ are the sample mean vectors of dimension 6 for the 6 distinct values of the vectorized data. Let $\Sigma_X$ and $\Sigma_Y$ be the sample covariance matrices.  For testing the two-sample hypothesis $H_0:$ $Q_1=Q_2$,  use the test statistic $(\bar X-\bar Y)\Sigma^{-1}(\bar X-\bar Y)^T$ with  $\Sigma=(1/n_1\Sigma_X+1/n_2\Sigma_Y)$, which has the asymptotic chisquare distribution $\chi^2(6)$.  A similar test statistic is used for the log-Euclidean distance, after taking matrix-log of the data.

$\sym^{+}(3)$, the space of  $3\times 3$  positive definite matrices, has important applications  in diffusion tensor imaging (DTI). Diffusion tensor imaging provides  measurements of  $3\times 3$ diffusion matrices of molecules of water in tiny voxels in the white matter of the brain.  When there are no barriers, the diffusion matrix is isotropic.   When a trauma occurs, due to an injury or a disease, this highly organized structure,   due to axon (nerve fiber) bundles and their myelin sheaths (electrically insulating layers), is disrupted and anisotropy decreases. Statistical analysis of DTI data using  two- and multiple-sample tests is important  in  investigating brain diseases such as  autism, schizophrenia, Parkinson's disease and Alzheimer's disease. There has been a growing body of work on DTI data analysis  \cite{2014arXiv1407.6383S, 2014arXiv1406.3361J, dryden2009}.


We now  consider a diffusion tensor imaging (DTI) data set  consisting of 46 subjects with 28 HIV+ subjects and 18 healthy controls.   Diffusion tensors were extracted along the fiber tract of the splenium of the corpus callosum. The DTI data for all the subjects are  registered in the same \emph{atlas space} based on arc lengths, with 75 features  obtained along the fiber tract of each subject. This data set has been studied in a regression setting in \cite{Yuan2012}. Our results are new and do not follow from \cite{Yuan2012}.  We carry out two sample tests between the control group and the HIV+ group for each of the 75 sample points along the fiber tract. Therefore, 75 tests are performed  in total. Two types of tests are carried out based on the Euclidean distance and the log-Euclidean distance.  


The simple Bonferroni procedure for testing $H_0$ yields a $p$-value equal to 75 times the smallest $p$-value which is of order $10^{-7}$. To identify sites with significant differences, the 75 $p$-values are ordered from the smallest to the largest with a \emph{false discovery rate} of $\alpha=0.05$, $58$ sites are found to yield significant differences using the Euclidean distance, and 47 using the $\log$-Euclidean distance (see \cite{citeulike:1042553}).

%
%


\begin{remark}
Extremely small $p$-values such as of the order $O(10^{-5})$ or smaller, computed using the chisquare approximation, are subject to coverage errors. They simply indicate that the $p$-value is extremely small. With such large observed values of the statistic von Bahr's inequality \cite{bahr}, showing the tail probability under $H_0$ to be smaller than $o(n^{-r})$ for every $r>0$, may perhaps be used as a justification.

\end{remark}

\section*{Acknowledgement.}
The authors are grateful to the two referees for their reviews. Their constructive suggestions and criticism have helped us improve the paper.
 The authors are indebted to Professor Susan Holmes for a helpful discussion. We thank Professor Hongtu Zhu for kindly providing us the data sets used in Section \ref{sec-data}.  This work is partially supported by the NSF grants DMS 1406872 and IIS 1546331.


\bibliographystyle{amsplain}
\bibliography{reference}

\end{document}